\documentclass[a4paper,reqno,11pt]{amsart}
\usepackage{amsmath}
\usepackage{amssymb}
\usepackage{cases}
\allowdisplaybreaks[4]

\newtheorem{thm}{Theorem}[section]
\newtheorem{prop}[thm]{Proposition}

\newtheorem{lem}[thm]{Lemma}

\newtheorem{cor}[thm]{Corollary}

\theoremstyle{definition}
\newtheorem{definition}[thm]{Definition}

\theoremstyle{remark}

\numberwithin{equation}{section}

\newcommand{\bQ}{\mathbb{Q}}

\newcommand\PP{{\mathcal{P}}}

 \newcommand\DD{{\mathcal{D}}}
 \newcommand\TT{{\mathcal{T}}}
 \newcommand\UU{{\mathcal{U}}}

 \newcommand\Supp{{\rm{Supp}}}

\newcommand\Nklt{{\rm{Nklt}}}

\newcommand\Chow{{\rm{Chow}}}
\begin{document}

\title{On Fujita invariants of subvariaties of a uniruled variety}
\date{\today}
\author{Christopher D. Hacon}
\address{%
CDH:\newline\indent Department of Mathematics \\  
University of Utah\\  
Salt Lake City, UT 84112, USA}
\email{hacon@math.utah.edu}

\author{Chen Jiang}
\address{%
CJ:\newline\indent Kavli IPMU (WPI)\\ UTIAS, The University of Tokyo\\ Kashiwa, Chiba 277-8583, Japan.}
\email{chen.jiang@ipmu.jp}
 \thanks{The first author
 was supported by NSF research grants no: DMS-1300750, DMS-1265285 and by
  a grant from the Simons Foundation; Award Number: 256202.}
\thanks{The second author was supported by the Program for Leading Graduate  Schools and World Premier International Research Center Initiative (WPI), MEXT, Japan}

\begin{abstract}
We show that if $X$ is a smooth uniruled projective variety  and $L$ a big and semiample $\bQ$-divisor on $X$, then there exists a proper closed subset $W\subset X$ such that every subvariety $Y$ satisfying
$a(Y,L)>  a(X,L)$ is contained in $W$. 
\end{abstract} 

\maketitle

\section{Introduction}

If $X$ is a smooth projective variety and $L$ is a  big $\bQ$-divisor on $X$, then the {\it Fujita invariant}, or {\it $a$-constant} is defined as follows
$$a(X,L)={\rm inf}\{ t>0 \mid  K_X+tL\ {\rm is\ big}\}.$$
Note that $a(X,L)\in \mathbb R_{\geq 0}$ is well defined since $K_X+tL$ is big for all $t>0$ sufficiently large, and that $a(X,L)>0$ if and only if $K_X$ is not pseudo-effective. It is easy to see that the $a$-constant is a birational invariant in the sense that if $\nu :X'\to X$ is a birational morphism of smooth varieties and $L'=\nu ^*L$, then $a(X,L)=a(X',L')$. Therefore we may also define the $a$-constant for a big $\bQ$-Cartier $\bQ$-divisor $L$ on an arbitrary normal projective variety $X$ by letting 
$$a(X,L):=a(X',L')$$ where $\nu :X'\to X$ is a resolution of singularities and $L'=\nu ^* L$. Note that if $X$ is smooth, then the $a$-constant is the usual pseudo-effective threshold, however if $X$ is singular, these numbers may be different.

In \cite{LTT}, motivated by a conjecture of Batyrev and Manin that relates arithmetic
properties of varieties with ample anticanonical class to geometric
invariants,
$a$-constants were intensively studied by Lehmann, Tanimoto and Tschinkel. They show that
(\cite[Theorem 1.1]{LTT}), if  $X$ is a smooth uniruled projective variety  and $L$ an ample $\bQ$-divisor on $X$, 
then there exists a countable union of proper closed subsets $W\subset X$ such that every subvariety $Y$ satisfying
$a(Y,L)>  a(X,L)$ is contained in $W$. 
For the purpose of applications, it is expected that one may choose $W$ to be a proper closed subset of $X$. The purpose of this note is to prove that this is indeed the case:
\begin{thm}\label{main}
Let $X$ be a smooth uniruled projective variety  and $L$ a big and semiample $\bQ$-divisor on $X$. 
Then there exists a proper closed subset $W\subset X$ such that every subvariety $Y$ satisfying
$a(Y,L)>  a(X,L)$ is contained in $W$. 
\end{thm}
Note that this result is proven in \cite[Theorem 1.2]{LTT} assuming that  a weak version of the BAB conjecture holds in dimension $n-1=\dim X -1$. We expect that Theorem \ref{main} holds also if we just assume that $L$ is nef and big (rather than big and semiample).

Our idea is to replace the WBAB conjecture assumed in  \cite[Theorem 1.2]{LTT} by constructing non-klt centers (see Proposition \ref{nklt}) and applying finiteness of $a$-constants (see Corollary \ref{finite}). This is an application of a recent result of Di Cerbo \cite{DC} based on a boundedness result proved by Birkar \cite{B}.

{\bf Acknowledgment.}  The essential ideas in this note originated during a visit of the second author to the University of Utah in February 2015. The second author is grateful for the hospitality. The authors would like to thank Brian Lehmann for useful comments.

\section{Preliminaries}
In this paper we work over the field of complex numbers $\mathbb C$.

\subsection{Facts on $a$-constants}
In this subsection, for the convenience of the reader,  we collect several facts about $a$-constants that were proven in \cite{LTT}.
\begin{prop}[{\cite[Proposition 4.1]{LTT}}] \label{4.1}
Let $X$ be a smooth projective variety and $L$ a big and nef $\bQ$-divisor. Let $\UU \to W$ be a family of subvarieties of $X$ such that $\UU\to X$ is dominant. Then a general member $Y$ of the family $\UU$ satisfies $a(Y, L)\leq    a(X, L)$.
\end{prop}
\begin{thm}[{\cite[Theorem 4.2]{LTT}}]\label{4.2}
Let $X$ be a smooth projective variety and $L$ a big and
nef $\bQ$-divisor. Let $\pi : \mathcal{U} \to W$ be a family of subvarieties of $X$. There
exists a proper closed subset $V\subset X$ such that if a member $Y$ of the
family $\mathcal{U} $ satisfies $a(Y, L) > a(X, L)$ then $Y \subset V$.
\end{thm}

\begin{prop}[{\cite[Proposition 4.6]{LTT}}]\label{4.6}
Let $X$ be a smooth uniruled projective variety and $L$ a big and nef $\bQ$-divisor. Then either
\begin{enumerate}
\item $X$ is covered by  proper subvarieties Y satisfying $a(Y, L) = a(X, L)$, or
\item $X$ is birational to a $\bQ$-factorial terminal Fano variety $X'$ of
Picard number $1$.
\end{enumerate}
\end{prop}

\begin{lem}[{\cite[Lemma 4.7]{LTT}}]\label{4.7}
Let $X$ be a smooth projective variety and $L$ a big and
nef $\bQ$-divisor on $X$. Fix a constant $C$. Then the subset of $\Chow(X)$
parametrizing subvarieties of $X$ that are not contained in ${\bf B}_+(L)$ and
are of $L$-degree at most $C$ is bounded.
\end{lem}


\subsection{Non-klt centers}
We follow the standard notation and conventions of the minimal model program, see eg. \cite{SOP}.
\begin{definition} Let $(X,\Delta )$ be a pair so that $X$ is a normal variety, $\Delta$ is an effective $\mathbb Q$-divisor, and $K_X+\Delta$ is $\mathbb Q$-Cartier.
We say that a subvariety $V\subset X$ is a {\it non-klt center} of $(X,\Delta)$ if it is
the image of a divisor of  discrepancy at most $-1$. We will denote
by $\Nklt(X, \Delta)$ the union of all non-klt centers of $(X,\Delta)$. A {\it non-klt place} is a valuation corresponding to a
divisor of discrepancy at most $-1$. A non-klt center is {\it pure} if $K_X + \Delta$ is log canonical at the generic point of $V$. If moreover there
is a unique non-klt place lying over the generic point of $V$, we
will say that $V$ is an {\it exceptional} non-klt center.
\end{definition}
The following is a weak form of Kawamata's subadjunction theorem.
\begin{thm}[Subadjunction, see {\cite[Proposition 5.1]{JXD}}]\label{subadj}
Let $V\subset X$ be a non-klt center  of a pair $(X,\Delta)$ which is lc at a general point of $V$. Let $\nu: {V^\nu}\to  V$ be the normalization.
Then there is an effective $\bQ$-divisor $\Delta_{{V}^\nu}$ on ${{V}^\nu}$ such that
$$\nu^*(K_X +\Delta)|_{V_\nu}\sim_\bQ  K_{{V}^\nu}+\Delta_{V^\nu}.$$
\end{thm}

We have the following connectedness lemma of Koll\'{a}r and Shokurov for the non-klt locus (cf.  
Shokurov \cite{Shokurov}, Koll\'{a}r \cite[17.4]{Kol92}).
\begin{thm}[Connectedness Lemma]
Let $f:X\rightarrow Z$ be a proper morphism of normal varieties with connected fibers and $D$  a $\bQ$-divisor such that $-(K_X+D)$ is $\bQ$-Cartier, $f$-nef, and $f$-big. Write $D=D^+-D^-$ where $D^+$ and $D^-$ are effective with no common components. If $D^-$ is $f$-exceptional (i.e. all of its components have image of codimension at least $2$), then ${\rm Nklt} (X,D)\cap f^{-1}(z)$ is connected for any $z\in Z$. 
\end{thm}
We can use the following proposition to construct non-klt centers.
\begin{prop}[{cf. \cite[Lemma 3.2]{Lai}}]\label{nklt}
Let $X$ be  a  $\bQ$-factorial terminal Fano variety of dimension $n$.  Assume $(-K_X)^n>(wn)^n$ for some positive rational number $w$. Then for every point $p\in X$ there is an effective $\bQ$-divisor $\Delta_p\sim_\bQ -\frac{1}{w}K_X$ such that  the unique minimal non-klt center $V_p \subset \Nklt(X, \Delta_p)$ containing $p$ is
exceptional. 
\end{prop} 
\begin{proof}Fix a  point $p$. Fix a  positive rational number $w'$ such that $(-K_X)^n>(w'n)^n>(wn)^n$. By \cite[6.7.1 Theorem]{SOP},  there is an effective $\bQ$-divisor $\Delta'_p\sim_\bQ -\frac{1}{w'}K_X$ such that $(X, \Delta'_p)$ is not lc at $p$. Take $0 < t \leq 1$ the unique rational number such that $(X,  t\Delta'_p)$ is log
canonical but not klt at $p$. By \cite[Proposition 3.2, Lemma 3.4]{Ambro}, we can find an effective $\bQ$-divisor
$M_p \sim_\bQ -\frac{1}{w'}K_X$ and some rational number $a > 0$ such that for any rational number $0 < \epsilon \ll 1$, the pair
$(X,(1-\epsilon) t\Delta'_p + \epsilon aM_p)$ has a unique minimal non-klt center $V_p$ passing through $p$
which is exceptional. Note that 
$$
(1-\epsilon) t\Delta'_p + \epsilon aM_p\sim_\bQ -\frac{(1-\epsilon) t+\epsilon a}{w'}K_X
$$
and $\frac{(1-\epsilon) t+\epsilon a}{w'}<\frac{1}{w}$ for $0 < \epsilon \ll 1$.
Since $-K_X$ is ample, by adding a $\bQ$-divisor $\bQ$-linearly equivalent to a multiple of $-K_X$ to $\Delta'_p$, we conclude that there exists an effective $\bQ$-divisor
$\Delta_p\sim_\bQ  -\frac{1}{w}K_X$   and $(X,\Delta_p)$ has a unique minimal non-klt center $V_p$
passing through $p$ which is exceptional.
\end{proof}
\begin{lem}\label{w2}
Keep the notation in Proposition \ref{nklt}. If $w>2$, then $\dim V_p>0$ for a general point $p$.
\end{lem}
\begin{proof}
Assume to the contrary that there exist $p_1\in  X$ such that $V_{p_1}=\{p_1\}$ and $p_2\in  X\backslash \Supp(\Delta_{p_1})$ such that $V_{p_2}=\{p_2\}$. Then $p_1$ and  $p_2$ are contained in $\Nklt(X, \Delta_{p_1}+\Delta_{p_2})$ and $p_2$ is isolated by construction. On the other hand, 
$$
-(K_X+ \Delta_{p_1}+\Delta_{p_2})\sim_\bQ \Big(1-\frac{2}{w}\Big)(-K_X)
$$
is ample. By the connectedness lemma, $\Nklt(X, \Delta_{p_1}+\Delta_{p_2})$ is connected, which is a contradiction.
\end{proof}

\subsection{Finiteness of $a$-constants}
We recall the main result of \cite{DC} in this subsection.
\begin{definition}Let $X$ be a normal projective variety and $H$ a big $\bQ$-divisor. We define the {\it pseudo-effective threshold} to be
$$\tau(X, H) := \inf\{t \geq 0 \mid K_X +tH \text{ is big}\}.$$
\end{definition}
Note that if $X$ is smooth, $a$-constant and pseudo-effective threshold just coincide.
\begin{definition}[cf. {\cite[Definition 3.1]{DC}}]Fix a positive integer $n$ and two positive real numbers $\epsilon$ and $\delta$.
We define $\DD_n(\epsilon, \delta)$ to be the set of lc pairs $(X, \Delta)$ such that:
\begin{enumerate}
\item $X$ is a normal projective variety of dimension $n$,
\item $\Delta$ is a big $\bQ$-divisor with coefficients $\geq \delta$, and
\item $(X, t\Delta)$ is $\epsilon$-lc and $K_X + t\Delta$ is pseudo-effective for some $0\leq t\leq 1$.
\end{enumerate}
\end{definition}

\begin{definition}[cf. {\cite[Definition 3.2]{DC}}] Fix a positive integer $n$ and two positive real numbers $\epsilon$ and $\delta$.
We define the set
$$\TT_n(\epsilon, \delta) := \{\tau(X, \Delta) \mid (X, \Delta) \in \DD_n(\epsilon, \delta)\}.$$
\end{definition}
\begin{thm}[{\cite[Corollary 3.6]{DC}}]\label{DC36}
Fix a positive integer $n$ and three positive real numbers $\epsilon$, $\delta$ and $\eta$.
Then the set $\TT_n(\epsilon, \delta)\cap [\eta, 1]$ is a finite set.
\end{thm}

To apply this theorem in our situation, we have the following corollary.
\begin{definition}Fix a positive integer $n$.
We define $\PP_n$ to be the set of pairs $(Y, L)$ such that:
\begin{enumerate}
\item $Y$ is a normal projective variety of dimension $n$,
\item $L$ is a base point free big Cartier divisor.
\end{enumerate}
\end{definition}

\begin{cor}\label{finite}Fix a positive integer $n$ and a positive real number $\eta$.
Then the set $$ \{a(Y, L) \mid (Y, L) \in \PP_n\}\cap [\eta, \infty)$$ is a finite set.
\end{cor}
\begin{proof}We may assume that $\eta \leq \frac{1}{4(n+1)}$. 

Firstly, we show that the set $$\{a(Y, L) \mid (Y, L) \in \PP_n\}\cap \Big[\eta, \frac{1}{2}\Big]$$ is a finite set. Take $(Y, L) \in \PP_n$ and assume that $a(Y, L)\in  [\eta, \frac{1}{2}]$. Note that $a(Y, \frac{1}{2}L)=2a(Y, L)\in  [2\eta, 1]$.
By taking a resolution, we may assume that $Y$ is smooth. In this case $a(Y, \frac{1}{2}L)=\tau(Y, \frac{1}{2}L)$. Replacing $L$ by a general element in $|L|$, we may assume that $L$ is irreducible and smooth. Moreover, $(Y, \frac{1}{2}L)$ is $\frac{1}{2}$-lc and $K_Y+\frac{1}{2}L$ is pseudo-effective, that is, $(Y, \frac{1}{2}L) \in \DD_n(\frac{1}{2}, \frac{1}{2})$. This implies that the set
$$\Big\{ a\Big(Y, \frac{1}{2}L\Big) \Bigm| (Y, L) \in \PP_n\Big\}\cap [2\eta, 1]$$ is finite by Theorem \ref{DC36}, and so is $\{a(Y, L) \mid (Y, L) \in \PP_n\}\cap [\eta, \frac{1}{2}]$.

Then we show that  the set $$\{a(Y, L) \mid (Y, L) \in \PP_n\}\cap \Big[\frac{1}{2}, \infty\Big)$$ is a finite set. Take $(Y, L) \in \PP_n$ and assume that $a(Y, L)\geq \frac{1}{2}$. By taking a resolution, we may assume that $Y$ is smooth. By \cite[Proposition 2.10]{LTT}, $a(Y, L)\leq n+1$. Now we consider $(Y, 2(n+1)L) \in \PP_n$. Note that $a(Y, 2(n+1)L)=\frac{1}{2(n+1)}a(Y, L)$, hence $a(Y, 2(n+1)L)\in[\frac{1}{4(n+1)}, \frac{1}{2}]$. By the first step, $a(Y, 2(n+1)L)$ belongs to a finite set. Hence $a(Y, L)$ belongs to a finite set.
\end{proof}
\section{Proof of Theorem \ref{main}}
We prove the following proposition suggested by B. Lehmann. 
\begin{prop}\label{bounded prop}
Fix a positive real number t. Let $X$ be a smooth projective variety and $L$ a big and semiample $\bQ$-divisor. Then there is a bounded family $\mathcal{U}$ of subvarieties of $X$ such that any subvariety $Y$ not contained in $\mathbf B _+(L)$, with $a(Y,L)>t$ is dominated by some members $Z$ of $\mathcal{U}$, such that $a(Z, L)= a(Y, L)$. 
\end{prop}

\begin{proof}
Note that  for a  subvariety $Y$ not contained in ${\bf B}_+(L)$, $L|_Y$ is nef and big, and so $a(Y,L)$ is well defined. Therefore we will only consider subvarieties not contained in ${\bf B}_+(L)$.

Replacing $L$ by some multiple, we may assume that $L$ is a base point free Cartier divisor. 

We construct $\UU$ inductively by increasing induction on the dimension of $Y$. 

For a subvariety $Y$ with $a(Y,L)>t$ and $\dim Y=1$, it is easy to see that $Y$ is a rational curve with $$\deg_Y(L)=Y\cdot L=\frac{2}{a(Y,L)}<\frac{2}{t}.$$ By Lemma \ref{4.7}, such $Y$ form a bounded family $\UU_1$.

Suppose that we have constructed a bounded family $\mathcal{U}_i$ of subvarieties such that every subvariety $Y$ with $a(Y,L)>t$ and $\dim Y\leq i$ is dominated by some members $Z$ of $\mathcal{U}$ such that $a(Z, L)= a(Y, L)$. 
We construct $\UU_{i+1}$ as follows. Suppose that $Y$ is an $(i+1)$-dimensional subvariety satisfying $a(Y,L)>t$. By taking a resolution, we may assume that $Y$ is smooth. Proposition \ref{4.6} shows that either
\begin{enumerate}
\item $Y$ is covered by proper subvarieties $Z$ with 
$a(Z,L)=a(Y,L),$
or
\item $Y$ is birational to a  $\bQ$-factorial terminal Fano variety $Y'$ of Picard number 1.
\end{enumerate}

In  Case (1), by induction, $Z$ is dominated by some members $Z'$ of $\mathcal{U}_i$ such that $a(Z', L) = a(Z, L)$, and so is $Y$.

In Case (2), by taking a resolution, we may assume $\phi: Y\dashrightarrow Y'$ is a morphism. By the proof of {\cite[Proposition 4.6]{LTT}}, $K_{Y'}+a(Y,L)\phi_*(L|_Y)\equiv 0$.

We define constant $c_0<1$ and $w>2$ as follows: since $L$ is base point free, we know that the set 
$$\{a(Z, L)\mid Z \text{ is a subvariety of } X\}\cap (t, \infty]$$ is finite
 by Corollary \ref{finite}. Hence we may take a rational number $c_0<1$ such that the set $$\{a(Z, L)\mid Z \text{ is a subvariety of } X\} \cap [c_0a(Z', L), a(Z', L))$$ is empty for any subvariety $Z'$ with $a(Z', L)>t$. Take $w=\frac{1}{1-c_0}$. We may assume $w>2$ by decreasing $c_0$.

If $(-K_{Y'})^{i+1}\leq (w(i+1))^{i+1}$, then 
$$
(L|_Y)^{i+1}\leq (\phi_*(L|_Y))^{i+1}\leq \frac{(w(i+1))^{i+1}}{a(Y, L)^{i+1}}<\frac{(w(i+1))^{i+1}}{a(X,L)^{i+1}}.
$$
Then by  Lemma \ref{4.7}, such $Y$ form a bounded family $\UU'_{i+1}$.

Now we assume that   $(-K_{Y'})^{i+1}> (w(i+1))^{i+1}$. By Proposition \ref{nklt}, for  a general point $p\in Y'$, there exists an effective $\bQ$-divisor $\Delta'_p\sim_\bQ -\frac{1}{w}K_{Y'}$ such that $V'_p\subset\Nklt({Y'}, \Delta'_p )$ is the minimal exceptional non-klt center containing $p$. 
Note that by Lemma \ref{w2} and $w>2$, $\dim V'_p>0$. Let $\nu:\tilde{V}^\nu_p\to {V}'_p$ be the normalization. For any $\bQ$-Cartier divisor $G$ on ${V}'_p$, we denote $G|_{\tilde{V}^\nu_p}=\nu ^*G$.
By Theorem \ref{subadj},   there is an effective $\bQ$-divisor $\Delta_{\tilde{V}^\nu_p}$ such that
$$
(K_{Y'}+\Delta'_p)|_{\tilde{V}^\nu _p}\sim_{\bQ} K_{\tilde{V}^\nu_p}+\Delta_{\tilde{V}^\nu_p}.
$$
Note that since $K_{Y'}+a(Y,L)\phi_*L\equiv 0$, we have
$$
 K_{\tilde{V}^\nu_p}+\Delta_{\tilde{V}^\nu_p}+\Big(1-\frac{1}{w}\Big)a(Y,L)\phi_*L|_{\tilde{V}^\nu_p}\sim_\bQ 0.
$$
Let $V_p$ be the strict transform of $V'_p$ on $Y$. Let $\tilde{V}_p$ be a common   resolution of ${\tilde{V}^\nu_p}$ and $V_p$, $f: \tilde{V}_p\to V_p$, $g: \tilde{V}_p\to \tilde{V}^\nu_p$. Then
\begin{align*}
{}&K_{\tilde{V}_p}+\Big(1-\frac{1}{w}\Big)a(Y,L)f^*(L|_{V_p})\\
=\ {}&g^* \Big(K_{\tilde{V}^\nu_p}+\Delta_{\tilde{V}^\nu_p}+\Big(1-\frac{1}{w}\Big)a(Y,L)\phi_*L|_{\tilde{V}^\nu_p}\Big)-g_*^{-1}\Delta_{\tilde{V}^\nu_p}+E\\
\sim_\bQ {}&-g_*^{-1}\Delta_{\tilde{V}^\nu_p}+E,
\end{align*}
where  $E$ is a $g$-exceptional $\bQ$-divisor. Note that the $\bQ$-divisor $-g_*^{-1}\Delta_{\tilde{V}^\nu_p}+E$ is not big.
Hence $K_{\tilde{V}_p}+(1-\frac{1}{w})a(Y,L)f^*(L|_{V_p})$ is not big and therefore 
$$
a(V_p, L)\geq \Big(1-\frac{1}{w}\Big)a(Y,L)={c_0}a(Y,L).$$
By the definition of $c_0$, this implies that $
a(V_p, L)\geq  a(Y,L).$
Since $p$ is a general point, $Y$ is dominated by such $V_p$.  By induction, $V_p$ is dominated by some members $Z$ of $\mathcal{U}_i$ such that $a(Z, L) =a(V_p, L) \geq a(Y, L)$. Hence $Y$ is dominated by some members $Z$ of $\mathcal{U}_i$ such that $a(Z, L) \geq a(Y, L)$. By Proposition \ref{4.1}, by taking general members, $Y$ is dominated by some members $Z$ of $\mathcal{U}_i$ such that $a(Z, L) = a(Y, L)$.

Hence we may take $\UU_{i+1}=\UU_i\cup \UU'_{i+1}$, and the proof is completed.
\end{proof}

\begin{proof}[Proof of Theorem \ref{main}]
Take $t=a(X, L)$ in Proposition \ref{bounded prop}, there is a bounded family $\mathcal{U}$ of subvarieties of $X$ such that any subvariety $Y$ not contained in $\mathbf B _+(L)$, with $a(Y,L)>a(X, L)$ is dominated by some members $Z$ of $\mathcal{U}$, such that $a(Z, L)= a(Y, L)>a(X, L)$. By Theorem \ref{4.2}, there exists a proper closed subset $W \subset X$ such that  any member $Z$ of the family $\UU$ satisfying $a(Z, L) >a(X, L)$ is contained in $W$. Hence any subvariety $Y$ with $a(Y,L)>a(X, L)$ is contained in $W$.
\end{proof}


\begin{thebibliography}{99}
\bibitem{Ambro} F. Ambro, {\em The locus of log canonical singularities}, arXiv:math/9806067.


\bibitem{B} C. Birkar, {\it Anti-pluricanonical systems on Fano varieties},  arXiv:1603.05765.

\bibitem{DC}G. Di Cerbo, {\it On Fujita's spectrum conjecture}, arXiv:1603.09315.


\bibitem{JXD} X. Jiang, {\it On the pluricanonical maps of varieties of intermediate Kodaira dimension}, Math. Ann. {\bf 356} (2013), 979--1004.


\bibitem{SOP} J. Koll\'ar, {\em Singularities of pairs}, Algebraic Geometry, Santa Cruz 1995, Proc. of Symp. in Pure Math. Vol. 62., Amer. Math. Soc., 1997.
\bibitem{Kol92} J. Koll\'ar, et al, {\em Flips and abundance for algebraic threefolds}, A summer
seminar at the University of Utah, Salt Lake City, 1991, Ast\'{e}risque, {\bf 211} (1992)


\bibitem{Lai} C-J. Lai, \emph{Bounding the volumes of singular Fano
threefolds}, Nagoya Math. J., to appear.


 \bibitem{LTT} B. Lehmann, S. Tanimoto, Y. Tschinkel, \emph{Balanced line bundles on Fano varieties},  J. Reine Angew. Math., to appear.
 
 
\bibitem{Shokurov} V. V. Shokurov, {\em 3-fold log flips}, Izv. A. N. SSSR, Ser. Math. {\bf 56} (1992), 105--201 \& {\bf 57} (1993), 141--175; English transl. Russian Acad. Sci. Izv. Math.
{\bf 40} (1993), 93--202 \& {\bf 43} (1994), 527--558.
 

\end{thebibliography}
\end{document}